\documentclass[12pt,reqno]{amsart}
\usepackage{amsmath,amssymb,amsfonts,amscd,latexsym,amsthm,mathrsfs,verbatim}
\usepackage[colorlinks,linkcolor=black,citecolor=black]{hyperref}
\usepackage{hypbmsec}
\usepackage{bm}
\usepackage{mathtools}
\usepackage{tikz}
\usetikzlibrary{matrix,arrows,decorations.pathmorphing}
\usepackage[margin=1in]{geometry}
\newtheorem{lemma}{Lemma}
\newtheorem{theorem}{Theorem}

\newtheorem{prop}{Proposition}
\theoremstyle{remark}
\newtheorem{remark}{\bf Remark}

\renewcommand{\Re}{\operatorname{Re}}

\newcommand{\Z}{\mathbb{Z}}

\newcommand{\C}{\mathbb{C}}
\newcommand{\N}{\mathbb{N}}

\usepackage{etoolbox}
\patchcmd{\section}{\scshape}{\bfseries}{}{}
\makeatletter
\renewcommand{\@secnumfont}{\bfseries}
\makeatother
\numberwithin{equation}{section}
\numberwithin{lemma}{section}
\numberwithin{theorem}{section}
\numberwithin{prop}{section}
\numberwithin{remark}{section}

\begin{document}

\title{Uniform bounds for sums of Kloosterman sums of half integral weight}

\author{Alexander Dunn}
\address{Department of Mathematics, University of Illinois, 1409 West Green
Street, Urbana, IL 61801, USA}
\email{ajdunn2@illinois.edu}

\maketitle
\begin{abstract}
For $m,n>0$ or $mn<0$ we estimate the sums
\begin{equation*}
\sum_{c \leq x} \frac{S(m,n,c,\chi)}{c},
\end{equation*}
where the $S(m,n,c,\chi)$ are Kloosterman sums attached to a multiplier $\chi$ of weight $1/2$ on the full modular group. Our estimates are uniform in $m, n$ and $x$ in analogy with the bounds for the case $mn<0$ due to Ahlgren--Andersen, and those of Sarnak--Tsimerman for the trivial multiplier when $m,n>0$. In the case $mn<0$, our estimates are stronger in the $mn$-aspect than those of Ahlgren--Andersen. We also obtain a refinement whose quality depends on the factorization of $24m-23$ and $24n-23$ as well as the best known exponent for the Ramanujan--Petersson conjecture. 
\end{abstract}

\section{Introduction and statement of results }
The classical Kloosterman sum
\begin{equation*}
S(m,n,c):=\sum_{d \pmod c} e \Big( \frac{m d+n \overline{d}}{c} \Big), \quad e(x):=\text{exp}(2 \pi i x)
\end{equation*}
plays a central part in analytic number theory. For applications, see \cite{HB,S} for example.

In this paper we study generalised Kloosterman sums $S(m,n,c,\chi)$ attached to the Dedekind eta multiplier $\chi$ of weight $1/2$. These are given by 
\begin{equation} \label{genklo}
S(m,n,c,\chi):=\sum_{\substack{0 \leq a,d<c \\ ( \begin{smallmatrix} 
a & b \\
c & d
\end{smallmatrix} ) \in \text{SL}_2(\mathbb{Z})}} \overline{\chi}  \begin{pmatrix}
 a & b \\
 c & d 
 \end{pmatrix} e \Big( \frac{\tilde{m}a +\tilde{n}d }{c} \Big), 
\end{equation}
where 
\begin{equation*}
\tilde{m}:=m-\frac{23}{24}, \quad m \in \mathbb{Z}.
\end{equation*}

The Kloosterman sums defined in \eqref{genklo} appear in Rademacher's formula for the partition function $p(n)$ \cite{R1,R2}. Cancellation amongst such sums plays a central role in establishing power saving error terms when one truncates of the formula for $p(n)$. This can be found in the work of Ahlgren and Andersen \cite{AA}. The study of closely related Kloosterman sums has applications to the coefficients of Ramanujan's well known mock theta function $f(q)$ as well. This can be found in the work of Ahlgren and the author \cite{AD}. 

Kloosterman sums with general multipliers have been studied by Bruggeman \cite{Br}, Goldfeld--Sarnak \cite{GS} and Pribitkin \cite{Pr}, amongst many others.

For the ordinary Kloosterman sums $S(m,n,c)$, Linnik \cite{Li} and Selberg \cite{Se} conjectured that there should be considerable cancellation in the sums 
\begin{equation} \label{classickl}
\sum_{c \leq x} \frac{S(m,n,c)}{c}.
\end{equation}
Sarnak and Tsimerman \cite{ST} proposed a modified version of Linnik's and Selberg's conjecture with an $\varepsilon$-``safety valve" in $m$ and $n$. In particular, the refined conjecture for \eqref{classickl} is
\begin{equation*}
\sum_{c \leq x} \frac{S(m,n,c)}{c} \ll_{\varepsilon} (|mn| x)^{\varepsilon}.
\end{equation*} 
One obtains the ``trivial bound" for \eqref{classickl} by applying Weil's bound \cite{W}
\begin{equation*}
|S(m,n,c)| \leq \tau(c) (m,n,c)^{\frac{1}{2}} \sqrt{c},
\end{equation*}
where $\tau(c)$ is the number of divisors of $c$. This yields 
\begin{equation*}
\sum_{c \leq x} \frac{S(m,n,c)}{c} \ll \tau((m,n)) x^{\frac{1}{2}} \log x.
\end{equation*}
Still the best known bound in the $x$ aspect was obtained by Kuznetsov \cite{Ku}, who proved for $m,n>0$ that
\begin{equation} \label{Ku}
\sum_{c \leq x} \frac{S(m,n,c)}{c} \ll_{m,n} x^{\frac{1}{6}} (\log x)^{\frac{1}{3}}. 
\end{equation}
Sarnak and Tsimerman \cite{ST} refined Kuznetsov's method and made the dependence on $m$ and $n$ explicit. They proved that for $m,n>0$ we have
\begin{equation} \label{STb}
\sum_{c \leq x} \frac{S(m,n,c)}{c} \ll \big(x^{\frac{1}{6}}+(mn)^{\frac{1}{6}}+(m+n)^{\frac{1}{8}} (mn)^{\frac{\theta}{2}} \big)(xmn)^{\varepsilon},
\end{equation}
where $\theta$ is any admissible exponent in the Ramanujan--Petersson conjecture for the coefficients of weight zero Maass cusp forms.  By work of Kim and Sarnak \cite[Appendix~2]{Ki}, the exponent $\theta=7/64$ is available. Ganguly and Sengupta \cite{GSe}  have generalised the results of Sarnak and Tsimerman to sums over $c$ that are divisible by a fixed integer $q$.

K\i ral \cite{Kir} obtained estimates in the case $mn<0$ using the opposite sign Kloosterman zeta function. He obtained the bound
\begin{equation*}
\sum_{c \leq x} \frac{S(m,n,c)}{c} \ll x^{\frac{1}{6}+\varepsilon} \big( (m,n)^{\varepsilon}+(mn)^{\theta} \big)+x^{\varepsilon} (mn)^{\frac{1}{4}+\varepsilon},
\end{equation*}
where $\theta$ is as above.

For the Kloosterman sum $S(m,n,c,\chi)$, Ahlgren and Andersen \cite[Theorem~1.3]{AA} proved that for $mn<0$ we have
\begin{equation} \label{AAeta}
\sum_{c \leq x} \frac{S(m,n,c,\chi)}{c} \ll_{\varepsilon}  \big(x^{\frac{1}{6}}+|mn|^{\frac{1}{4}} \big) |mn|^{\varepsilon} \log x.
\end{equation}
They obtained a stronger result for sums of Kloosterman sums $S(1,n,c,\chi)$ when $n<0$ \cite[Theorem~9.1]{AA}. This leads to an improvement in error term \cite[Theorem~1.1]{AA} when one truncates Rademacher's formula \cite{R1,R2} for the partition function $p(n)$. 

Our first Theorem improves the $mn$-aspect in \eqref{AAeta}.
\begin{theorem} \label{mainthm}
Let $m>0$ and $n<0$ be integers such that $24n-23$ is not divisible by $5^4$ or $7^4$. Then we have 
\begin{equation*}
\sum_{c \leq x} \frac{S(m,n,c,\chi)}{c} \ll_{\varepsilon}  \Big(x^{\frac{1}{6}}+\frac{m^{\frac{1}{4}}+|n|^{\frac{1}{4}}}{|mn|^{\frac{1}{308}}}+|mn|^{\frac{19}{77}} \Big) |mn|^{\varepsilon} \log^3 x.
\end{equation*}
\end{theorem}

We also consider the case when $m,n>0$. Let $\mathcal{P}:=\big \{\frac{k(3k \pm 1)}{2} : k \in \mathbb{Z} \big \}$ be the set of generalized pentagonal numbers.

\begin{theorem} \label{mainthm2}
Let $m,n>0$ be integers be such that $m-1 \not \in \mathcal{P}$ or $n-1 \not \in \mathcal{P}$. Then we have
\begin{equation*}
\sum_{c \leq x} \frac{S(m,n,c,\chi)}{c} \ll_{\varepsilon}  \big(x^{\frac{1}{6}}+(mn)^{\frac{1}{4}} \big) (mn)^{\varepsilon} \log^2 x.
\end{equation*}
\end{theorem}

\begin{remark}
When both $m,n>0$ are such that $m-1 \in \mathcal{P}$ and $n-1 \in \mathcal{P}$, we have the asymptotic formula 
\begin{equation*}
\sum_{c \leq x} \frac{S(m,n,c,\chi)}{c}=C(m,n) x^{\frac{1}{2}}+O_{m,n} (x^{\frac{1}{6}}),
\end{equation*}
for some constant $C(m,n)$. See \cite[Theorem~8]{AA2}.
\end{remark}

We also obtain refined bounds which recognise the arithmetic of $24m-23$ and $24n-23$. In analogy with the result of Sarnak and Tsimerman, these depend on progress toward the Ramanujan--Petersson conjecture for weight zero Maass forms on $\Gamma_0(N)$ (cf. the $H_{\theta}$-hypothesis in Section \ref{heckemass}). 

\begin{theorem} \label{auxthm2}
Let $m,n>0$ be integers such that $m-1 \not \in \mathcal{P}$ or $n-1 \not \in \mathcal{P}$. Suppose $m_0,n_0$ are integers such that $24m-23=m_0^2 s$ and $24n-23=n_0^2 t$ with $s$ and $t$ square-free integers. Then
\begin{equation*}
\sum_{c \leq x} \frac{S(m,n,c,\chi)}{c}  \ll_{\varepsilon} \Big( x^{\frac{1}{6}}+ (st)^{\frac{1}{4}}+(st)^{\frac{1}{12}} (mn)^{\frac{1}{6}}+ (mnst)^{\frac{1}{8}+\frac{\theta}{4}}  \Big) (mn)^{\varepsilon} \log^3 x,
\end{equation*}
where $\theta$ is any admissible exponent toward the Ramanujan--Petersson conjecture.
\end{theorem}

The proofs of Theorems \ref{mainthm}--\ref{auxthm2} depend on generalisations of Kuznetsov's trace formula due to Proskurin \cite{P} and Ahlgren--Andersen \cite{AA}. These are given in Sections \ref{proskurin} and \ref{AAkuzsec}. This formula transfers the task at hand to that of establishing bounds for sums involving the coefficients of half integral weight holomorphic and Maass cusp forms.   

To bound the contribution from holomorphic forms we appeal to Petersson's formula. We also use the Shimura lift for half integer weight forms and Deligne's bound to obtain bounds in terms of the factorisation of $24m-23$ and $24n-23$. Details can be found in Sections \ref{heckeholom} and \ref{holbd}. 

To bound the contribution from the Maass cusp forms we modify a dyadic argument in the spectral parameter that appears in \cite{AA,ST}. Our new treatment involves estimating an initial segment for the spectral parameter using an averaged bound of Duke \cite{D}. We will also make use of a mean value estimate for the coefficients of Maass cusp forms due to Andersen and Duke \cite{And}. These bounds can be found in Section \ref{massbd}. A Shimura lift for half integral weight Maass cusp forms developed by Ahlgren and Andersen is applied to obtain the dependence on $\theta$ in Theorem \ref{auxthm2}. This lift appears in Section \ref{heckemass}. 

The proofs of Theorems \ref{mainthm}--\ref{auxthm2} can be found in Sections \ref{mainsec}--\ref{auxsec} respectively. We follow the methods and presentation in \cite{AA}.

\section{Preliminaries} \label{pre}
We give only a concise background related to congruence subgroups. More details can be found in \cite{AA} and \cite{DFI} for example. See also \cite[Section~2]{AD}. Let $\mathbb{H}$ denote the upper-half plane. We have the usual action of $\text{SL}_2(\mathbb{R})$ on $\mathbb{H}$ given by
\begin{equation*}
\gamma \tau =\frac{a \tau +b}{c \tau +d}, \quad \text{for} \quad \tau \in \mathbb{H} \quad \text{and} \quad \gamma= \begin{pmatrix}
 a & b \\
 c & d
 \end{pmatrix} \in \text{SL}_2(\mathbb{R}).
\end{equation*}
For $\gamma \in \text{SL}_2(\mathbb{R})$ we define the weight $k$ slash operator for real analytic forms by
\begin{equation*}
f \vert_k \gamma:=j(\gamma,z)^{-k} f(\gamma z), \quad j(\gamma,z):=\frac{cz+d}{|cz+d|}=e^{i \text{arg}(cz+d)},
\end{equation*}
where the argument is always chosen in $(-\pi,\pi]$. The weight $k$ Laplacian is defined by 
\begin{equation*}
\Delta_k:=y^2  \Big(\frac{\partial^2}{\partial x^2}+\frac{\partial^2}{\partial y^2}  \Big)-iky \frac{\partial }{\partial x}.
\end{equation*}

For simplicity we will work only with the groups  $\Gamma_0(N)$ for $N \in \N$ and with weights $k \in \frac{1}{2} \mathbb{Z}$,
although much of what is said here holds in more generality. Let $\Gamma$ denote such a group.
 We say that $\nu: \Gamma \rightarrow \C^{\times}$ is a multiplier system of weight $k$ if 
\begin{itemize}
\item $|\nu|=1$
\item $\nu(-I)=e^{-\pi i k}$, and 
\item $\nu(\gamma_1 \gamma_2) j(\gamma_1 \gamma_2,\tau)^k=\nu(\gamma_1) \nu(\gamma_2) j(\gamma_2,\tau)^k j(\gamma_1, \gamma_2 \tau)^k$ for all $\gamma_1,\gamma_2 \in \Gamma$.
\end{itemize}

Given a cusp $\mathfrak{a}$, let $\Gamma_{\mathfrak{a}}:=\{\gamma \in \Gamma: \gamma \mathfrak{a}=\mathfrak{a} \}$ denote the  
stabilizer in $\Gamma$ and let $\sigma_{\mathfrak{a}}$ denote the unique (up to translation on the right) matrix in 
$\text{SL}_2(\mathbb{R})$ satisfying $\sigma_{\mathfrak{a}} \infty=\mathfrak{a}$ and
 $\sigma_{\mathfrak{a}}^{-1} \Gamma_{\mathfrak{a}} \sigma_{\mathfrak{a}}=\Gamma_{\infty}$. 
 Define $\alpha_{\nu,\mathfrak{a}} \in [0,1)$ by the condition 
\begin{equation*}
\nu \left( \sigma_{\mathfrak{a}} \begin{pmatrix}
 1 & 1 \\
 0 & 1  
 \end{pmatrix}
 \sigma_{\mathfrak{a}}^{-1} \right)=e \left({-\alpha_{\nu,\mathfrak{a}}} \right).
\end{equation*}
The cusp $\mathfrak{a}$ is singular with respect to  $\nu$ if $\alpha_{\nu,\mathfrak{a}}=0$.
When $\mathfrak a=\infty$ we suppress the subscript.

If $\nu$ is multiplier of weight $k$, then it is a multiplier in any weight $k^{\prime} \equiv k \pmod{2}$, and 
 $\bar \nu$ is a multiplier of weight $-k$. 
If $\alpha_{\nu}=0$ then  $\alpha_{\bar{\nu}}=0$, while 
 if $\alpha_{\nu}>0$ then $\alpha_{\bar{\nu}}=1-\alpha_{\nu}$.
For $n \in \Z$ we define
\begin{equation*}
n_\nu:=n-\alpha_{\nu};
\end{equation*}
then we have 
\begin{equation} \label{eq:n_nu_conj}
n_{\bar{\nu}}=\begin{cases} 
			-(1-n)_\nu\quad& \text{if $\alpha_\nu \neq 0$},\\
			n\quad&\text{if $\alpha_\nu= 0$}.
		\end{cases}
\end{equation}	

With this notation we  define the generalized Kloosterman sum (at the cusp $\infty$) by
\begin{equation}\label{eq:kloos_def}
	S(m,n,c,\nu) := \sum_{\substack{0\leq a,d<c \\ \gamma= \left(\begin{smallmatrix} 
	 a & b \\
	 c & d 
	 \end{smallmatrix} \right)
	 \in \Gamma}} \bar\nu(\gamma) e \left(\frac{m_\nu a+n_\nu d}{c} \right).
\end{equation}
We have the relationships
\begin{equation}\label{eq:nuconj}
\overline{S(m,n,c,\nu)}=
	\begin{cases}
	S(1-m, 1-n, c, \overline\nu)&\quad\text{if $\alpha_{\nu}>0$,}\\
	S(-m, -n, c, \overline\nu)&\quad \text{if $\alpha_{\nu}=0$.}
	\end{cases}
\end{equation}	

A function $f:\mathbb{H} \rightarrow \mathbb{C}$ is automorphic of weight $k$ and multiplier $\nu$ for $\Gamma_0(N)$ if 
\begin{equation*}
f \vert_k \gamma =\nu(\gamma)  f, \quad \text{for all} \quad \gamma \in \Gamma_0(N).
\end{equation*}
Let $\mathcal{A}_k(N,\nu)$ denote the space of such functions. If $f \in \mathcal{A}_k(N,\nu)$ is a smooth eigenfunction of $\Delta_k$ which satisfies the growth condition
\begin{equation*}
f(\tau) \ll y^{\sigma}+y^{1-\sigma}, 
\end{equation*}
for some $\sigma$ and all $\tau \in \mathbb{H}$, then it is called a Maass form. Let 
\begin{equation*}
\mathcal{L}_k (N,\nu):=\{f \in \mathcal{A}_k(N,\nu): \| f\| < \infty \},
\end{equation*}
where the norm is induced by the Petersson inner product 
\begin{equation*}
\langle f,g \rangle:=\int_{\mathbb{H} \backslash \Gamma_0(N)} f(\tau) \overline{g(\tau)} d \mu, \quad d \mu:=\frac{dx dy}{y^2}.
\end{equation*} 
Let $\mathcal{B}_k(N, \nu)$ denote the subspace of $\mathcal{L}_k(N, \nu)$ consisting of smooth functions $f$ such that $f$ and $\Delta_k f$ are bounded on $\mathbb{H}$. For all $f,g \in \mathcal{B}_k(N,\nu)$ we have 
\begin{equation*}
\langle \Delta_k f,g \rangle=\langle f, \Delta_k g \rangle. 
\end{equation*}
Furthermore, for any $f \in \mathcal{B}_k(N,\nu)$ we have 
\begin{equation*}
\langle f,-\Delta_k f \rangle \geq \frac{|k|}{2} \left(1-\frac{|k|}{2} \right) \geq 0.
\end{equation*}
Thus by a theorem of Friedrichs, the operator $-\Delta_k$  has a unique self-adjoint extension to $\mathcal{L}_k(N, \nu)$ (which we also denote $-\Delta_k$). Then by a theorem of von Neumann, the space $\mathcal{L}_k(N,\nu)$ has a complete spectral resolution with respect $-\Delta_k$, which we describe in detail now. There is both a continuous and discrete spectrum. For each singular cusp $\mathfrak a$ (and only at such cusps) 
there is an Eisenstein series $E_{\mathfrak a}(z, s)$.  These provide the continuous spectrum on the line $\Re s=\frac{1}{2}$, which covers  $[1/4, \infty)$. 

The reminder of the spectrum is discrete. It is countable and of finite multiplicity (with $\infty$ being the only limit point). We denote it by
\begin{equation*}
\lambda_0\leq \lambda_1\leq \dots
\end{equation*}
where we have
\begin{equation*}
\lambda_0\geq \frac{|k|}{2} \left(1-\frac{|k|}{2} \right).
\end{equation*}
One component of the discrete spectrum is provided by residues of the Eisenstein series 
$E_{\mathfrak a}(z,s)$ at possible simple poles $s$ with $\frac12<s\leq 1$;
the corresponding  eigenvalues  have $\lambda<\frac{1}{4}$. The remainder of the discrete spectrum
arises from \emph{Maass cusp forms}. We give more details about the discrete spectrum in what follows.

Denote by $\tilde{\mathcal{L}}_k(N, \nu)$ the subspace of $\mathcal{L}_k(N, \nu)$ spanned by the eigenfunctions of $\Delta_k$. If $f \in \tilde{\mathcal{L}}_k(N, \nu)$ has Laplace eigenvalue $\lambda$, then we write (here $0 \leq k<2$)
\begin{equation*}
\lambda=\frac{1}{4}+r^2,\qquad r\in  i \Big(0, \sqrt{1/4-(|k|/2) \big(1-|k|/2 \big)} \Big] \cup [0, \infty),
\end{equation*}
and refer to $r$ as the spectral parameter of $f$.  Denote by $\tilde{\mathcal{L}}_k(N, \nu, r)$  the subspace of such functions. Let $W_{\kappa, \mu}$ denote the usual $W$-Whittaker function (cf. \cite[Section~13.14]{DL}). Then each  $f \in \tilde{\mathcal{L}}_k(N, \nu, r)$ has a Fourier expansion of the form
\begin{equation}\label{eq:f_fourier}
f(\tau)=c_{0}(y) +  \sum_{n_\nu\neq 0} \rho(n) W_{\frac{k \text{sgn}(n_{\nu})}{2}, ir}(4\pi |n_\nu| y)e(n_\nu x),
\end{equation}
where
\begin{equation*}
c_{0}(y)=\begin{cases} 0\quad&\text{if $\alpha_\nu\neq 0$},\\
				      0\quad&\text{if $\alpha_\nu=0$ and $r\geq 0$,}\\
				      \rho(0)y^{\frac12+i r}&\text{if $\alpha_\nu=0$ and $r\in i(0, 1/4]$,}
		\end{cases}
\end{equation*}			      	
with coefficients $\rho(n)$. Note that in the last case, we have $\rho(0) \neq 0$ only when $f$ arises as a residue. Let $\mathcal{S}_k(N,\nu) \subseteq \tilde{\mathcal{L}}_k(N, \nu)$ be the subspace spanned by the Maass cusp forms (i.e. $c_0(y)=0$).

Two important multipliers of weight $\frac{1}{2}$ are 
 the  eta-multiplier $\chi$ on $\text{SL}_2(\mathbb{Z})$, given by 
\begin{equation}\label{eq:etamult}
\eta(\gamma\tau)=\chi(\gamma)\sqrt{c\tau+d}\,\eta(\tau), \qquad \gamma=\begin{pmatrix}
 a & b \\
 c & d
 \end{pmatrix} \in \text{SL}_2(\Z),
\end{equation}
and the theta-multiplier $\nu_\theta$ on $\Gamma_0(4)$, given by 
\begin{equation}\label{eq:thetamult}
\theta(\gamma\tau)=\nu_\theta(\gamma)\sqrt{c\tau+d}\,\theta(\tau), \qquad \gamma=\begin{pmatrix}
 a & b \\
 c & d
 \end{pmatrix} \in \Gamma_0(4).
\end{equation}
Here $\eta(\tau)$ and $\theta(\tau)$ are the two fundamental theta functions
\[\begin{aligned}
\eta(\tau)&:=q^\frac1{24}\prod_{n=1}^\infty(1-q^n),\\
\theta(\tau)&:=\sum_{n=-\infty}^\infty q^{n^2},
\end{aligned}
\]
where we use the standard notation
\begin{equation*}
q:=e(\tau)=e^{2\pi i\tau}.
\end{equation*}

For $\nu_\theta$ we have the formula
\begin{equation} \label{eq:def-theta-mult}
	\nu_\theta \begin{pmatrix}
	 a & b \\
         c & d
         \end{pmatrix} = \left( \frac{c}{d} \right) \epsilon_d^{-1},
\end{equation}
where $\left(\frac{\bullet}{\bullet} \right)$ is the extended  Kronecker symbol and 
\begin{equation*}
	\epsilon_d =
	\begin{cases}
		1 & \text{ if }d\equiv 1\pmod{4}, \\
		i & \text{ if }d\equiv 3 \pmod{4}.
	\end{cases}
\end{equation*}
From this we obtain
\begin{equation}\label{eq:thetaconj}
\bar{\nu_\theta}(\gamma)=\left( \frac{-1}{d} \right) \nu_\theta(\gamma), \qquad \gamma=\begin{pmatrix}
 a & b \\
 c & d
 \end{pmatrix} \in \Gamma_0(4).
\end{equation}

For  $c>0$ and $\gamma=\left(\begin{smallmatrix}
a & b \\
c & d
\end{smallmatrix} \right) \in \text{SL}_2(\mathbb{Z})$, we have another convenient formula  \cite[Section~4.1]{Kno}
\begin{equation} \label{kron}
	\chi(\gamma) = 
	\begin{dcases}
		\left( \frac{d}{c} \right)  e \left( \frac{1}{24} \left[(a+d)c-bd(c^2-1)-3c\right] \right) & \text{ if $c$ is odd}, \\
		\left(\frac{c}{d} \right)  e \left(\frac {1}{24} \left[(a+d)c-bd(c^2-1)+3d-3-3cd\right] \right) & \text{ if $c$ is even.}
	\end{dcases}
\end{equation}
We have $\chi \left( \begin{smallmatrix}
1 & b \\
0 & 1
\end{smallmatrix} \right)=e(\frac{b}{24})$. Finally, if $c>0$ we have $\chi(-\gamma)=i \chi(\gamma)$ (this follows since $\gamma$ and $-\gamma$ act the same way on $\mathbb{H}$).

When $(k,\nu)=(1/2,\chi)$ and we work on the full modular group $\Gamma$, there is neither continuous spectrum nor discrete spectrum arising from Eisenstein series because the only cusp is non-singular. This follows from the evaluation
\begin{equation*}
\chi \left( \begin{pmatrix}
1 & 1 \\
0 & 1
\end{pmatrix} \right)=e  \left( \frac{1}{24} \right).
\end{equation*}
Thus $\tilde{\mathcal{L}}_{\frac{1}{2}}(1,\chi)$ is spanned by Maass cusp forms. We fix an orthonormal basis $\{u_j\}$ with corresponding spectral parameters $r_j$ and with Fourier series given by 
\begin{equation} \label{fouriercusp}
u_j(\tau)=\sum_{n \neq 0} \rho_j(n) W_{\frac{\text{sgn}(n)}{4},ir_j}(4 \pi | \tilde{n}| y) e(\tilde{n} x),
 \end{equation}
 where 
 \begin{equation*}
 \tilde{n}:=n_{\chi}=n-\frac{23}{24}.
 \end{equation*}

\section{Hecke theory for holomorphic cusp forms} \label{heckeholom}
We briefly review Hecke theory for holomorphic cusp forms of half integral weight. For $N \in \mathbb{N}$ and $k \in 2 \mathbb{N}$, let $S_{\frac{1}{2}+k}(4N,\nu)$ denote the space of holomorphic cusp forms of weight $1/2+k$ on $\Gamma_0(4N)$ with a multiplier $\nu$ of weight $\frac{1}{2}$. Let $\Psi$ denote an even Dirichlet character mod $4N$. For all primes $p \nmid 4N$, the action of the Hecke operator $T_{p^2}$  on 
\begin{equation*}
f:=\sum_{r=1}^{\infty} a(r) e(r \tau)\in S_{\frac{1}{2}+k}(4N,\Psi \nu_{\theta}) 
\end{equation*}
is given in \cite[Theorem~1.7]{Sh} by
\begin{equation} \label{coeffrel}
T_{p^2}(f)(z):=\sum_{n=1}^{\infty} \Big( a_f(p^2 n)+\Psi^*(p) \Big( \frac{n}{p} \Big) p^{k-1} a_f(n)+ \Psi(p^2) p^{2k-1} a_f \big(n/p^2 \big) \Big) e(n \tau) \in S_{\frac{1}{2}+k}(4N,\Psi \nu_{\theta}),
\end{equation}
where $\Psi^*$ is the character modulo $4N$ defined by $\Psi^*(m)=\Psi(m) \left( \frac{-1}{m} \right)^k$ and $a(n/p^2)=0$ if $p^2 \nmid n$.

There are Hecke operators $T_{n^2}$ for all integers $n$ such that $(n,4N)=1$.  For $v \in \mathbb{N}$, the operators $T_{p^{2v}}$ are polynomials in the $T_{p^2}$. If $(nm,4N)=1$ and $(n,m)=1$ then
\begin{equation} \label{mult}
T_{n^2} T_{m^2}=T_{n^2 m^2}.
\end{equation}

Recall that $\chi$ is the Dedekind eta multiplier defined in \eqref{eq:etamult}. Letting $\chi_{12}:=\Big(\frac{12}{\bullet} \Big)$, we have the map 
\begin{equation} \label{image}
L: S_{\frac{1}{2}+k}(1,\chi) \rightarrow S_{\frac{1}{2}+k}(576,\chi_{12} \nu_{\theta}), \quad L(f)(\tau):=f(24 \tau).
\end{equation}
We justify the target space in \eqref{image} with the following computation.  Let $\Gamma_0(M,N)$ denote the subgroup of $\Gamma_0(N)$ consisting of matrices whose upper right entry is divisible by $M$. Equation \eqref{kron} and the remark following it show that 
\begin{equation*}
\chi(\gamma)=\left( \frac{c}{d} \right) e \left( \frac{d-1}{8} \right) \quad \text{for} \quad \gamma \in \Gamma_0(24,24).
\end{equation*}
This implies that 
\begin{equation*}
\chi \left( \begin{pmatrix}
a & 24b \\
c/24 &d 
\end{pmatrix} \right)=\left( \frac{12}{d} \right) \left( \frac{c}{d} \right) \epsilon_d^{-1} \quad \text{for} \quad \gamma=\begin{pmatrix}
a & b \\
c &d 
\end{pmatrix} \in \Gamma_0(576).
\end{equation*}

For primes $p \nmid 6$, we can define Hecke operators  $\tilde{T}_{p^2}$ on $S_{\frac{1}{2}+k}(1,\chi)$. Let 
\begin{equation*}
f(\tau):=\sum_{n=1}^{\infty} a(n) e \left( \bigg(n-\frac{23}{24} \bigg) \tau \right) \in S_{\frac{1}{2}+k}(1,\chi).
\end{equation*}
Then we can define the action of $\tilde{T}_{p^2}$ on $S_{\frac{1}{2}+k}(1,\chi)$ by
\begin{equation*}
\tilde{T}_{p^2} f=\sum_{n=1}^{\infty} \Big( a_f(p^2 n)+\Psi^*(p) \Big( \frac{n}{p} \Big) p^{k-1} a_f(n)+\Psi(p^2) p^{2k-1} a_f \big(n/p^2 \big) \Big) e \left( \bigg(n-\frac{23}{24} \bigg)  \tau \right).
\end{equation*}
Observe that 
\begin{equation} \label{commute}
L(\tilde{T}_{p^2} f )=T_{p^2}(Lf).
\end{equation}

We recall the Shimura correspondence for half-integral weight holomorphic cusp forms. 
\begin{lemma} \cite[Main Theorem]{Sh} and \cite[Proposition~5.1]{Cip}  \label{shim}
Let $N,k \in \mathbb{N}$ and $\Psi$ be a character modulo $4N$. Suppose that $g(\tau):=\sum_{n=1}^{\infty} a(n) e(n \tau) \in S_{k+\frac{1}{2}}(4N,\Psi \nu_{\theta})$. Let $t$ be a positive square-free integer, and define the Dirichlet character $\Psi_t$ by $\Psi_t(n):=\Psi(n) \big(\frac{-1}{n} \big)^{k} \big( \frac{t}{n} \big)$. Define $b_t(n) \in \mathbb{C}$ by 
\begin{equation*}
\sum_{n=1}^{\infty} \frac{b_t(n)}{n^s}:=L(s-k+1,\Psi_t) \sum_{n=1}^{\infty} \frac{a(tn^2)}{n^s}.
\end{equation*}
Then
\begin{equation*}
\emph{Sh}_t(g):=\sum_{n=1}^{\infty} b_t(n) e(n \tau) \in M_{2 k}(2N,\Psi^2).
\end{equation*}
Moreover, if $k \geq 2$, then $\emph{Sh}_t(g)$ is a cusp form. For all primes $p \nmid 4N$ and squarefree $t$, we have  
\begin{equation*}
\emph{Sh}_t \big(T_{p^2} g \big)=T_p \big( \emph{Sh}_t(g) \big),
\end{equation*}
where $T_p$ denotes the usual Hecke operator on $M_{2 k}(2N,\Psi^2)$.
\end{lemma}

\section{Hecke theory for Maass cusp forms} \label{heckemass}
We discuss Hecke theory for the spaces $\mathcal{S}_{0}(N,\mathbf{1})$ and $\mathcal{S}_{\frac{1}{2}}(1,\chi)$. For $(n,N)=1$, the Hecke operator on $\mathcal{S}_{0}(N,\mathbf{1})$ can be defined as 
\begin{equation*}
(\mathcal{T}_n f)(\tau)=\frac{1}{\sqrt{n}} \sum_{ad=n} \hspace{0.1cm} \sum_{b \hspace{-0.2cm} \mod{d}} f \left( \frac{a \tau+b}{d} \right) \in \mathcal{S}_0(N,\mathbf{1}).
\end{equation*}
We have 
\begin{equation*}
\mathcal{T}_m \mathcal{T}_n=\sum_{d \mid (m,n)} \mathcal{T}_{mnd^{-2}}.
\end{equation*}
We now record the explicit action of the Hecke operators on Fourier coefficients. If $f \in \mathcal{S}_{0}(N,\mathbf{1}) \cap \tilde{\mathcal{L}}_{0}(N,\mathbf{1},r)$ is a Maass cusp form with Fourier expansion 
\begin{equation}
f(\tau)=\sum_{n \neq 0} \rho(n) W_{0, ir}(4\pi |n | y)e(n x),
\end{equation} 
then 
\begin{equation*}
\mathcal{T}_p f=\sum_{n \neq 0} \left (p^{\frac{1}{2}} \rho(pn)+p^{-\frac{1}{2}} \rho \left(\frac{n}{p} \right) \right) W_{0, ir}(4\pi |n | y)e(n x).
\end{equation*}

Note that $\mathcal{T}_n$ commutes with $\Delta_0$, so $\mathcal{T}_n$ is an endomorphism of $\mathcal{S}_0(N,\mathbf{1}) \cap \tilde{\mathcal{L}}_0(N,\mathbf{1},r)$. Furthermore, for all $(n,N)=1$ we have 
\begin{equation*}
\langle \mathcal{T}_n f, g \rangle= \langle f, \mathcal{T}_n g \rangle.
\end{equation*}
Thus we can produce an orthonormal basis $\{v_j\}$ (each having spectral parameter $r_j$) for $\mathcal{S}_0(N,\mathbf{1})$ that consists of Hecke eigenforms for all $\mathcal{T}_n$ with $n$ coprime to $N$. Suppose that each $v_j$ has Fourier coefficients $\tilde{\rho}_j(n)$. The $H_{\theta}$--hypothesis asserts that 
\begin{equation*}
\lambda_j(n) \ll_{\varepsilon} n^{\theta+\varepsilon},
\end{equation*}
where the $\lambda_j(n)$ are the Hecke--Maass eigenvalues defined by
\begin{equation*}
\mathcal{T}_n v_j= \lambda_j(n) v_j.
\end{equation*}
The Ramanujan--Petersson conjecture asserts that $H_0$ is true. The best known result is due to Kim and Sarnak \cite[Appendix 2]{Ki}, who showed that the exponent $\theta=7/64$ is available.  
Applying $\mathcal{T}_{n}$ to the Fourier expansion of $v_j$ we see that \cite[(6.14),(6.15)]{DFI}
\begin{equation} \label{coeff}
\tilde{\rho}_j(n)=\lambda_j(|n|) \tilde{\rho}_j \big( \text{sgn}(n) \big) |n|^{-1/2}. 
\end{equation}

The Hecke operators $\mathcal{T}_{p^2}$ for $p \nmid 6$ are defined on $\mathcal{S}_{\frac{1}{2}}(1,\chi)$ \cite[Section~2.6]{AA} by
\begin{equation*}
T_{p^2} f=\frac{1}{p} \bigg[ \sum_{b \mod p^2} e \Big( \frac{-b}{24}  \Big) f  \vert_{\frac{1}{2}} \begin{pmatrix}
\frac{1}{p} & \frac{b}{p} \\
0 & p 
\end{pmatrix} + e \Big( \frac{p-1}{8} \Big) \sum_{h=1}^{p-1} e \Big(\frac{-hp}{24}  \Big) \Big( \frac{h}{p} \Big) f  \vert_{\frac{1}{2}} \begin{pmatrix}
 1 & \frac{h}{p} \\
0 & 1 
\end{pmatrix}  +f  \vert_{\frac{1}{2}} \begin{pmatrix}
 p & 0 \\
0 & \frac{1}{p} 
\end{pmatrix}   \bigg].
\end{equation*}
Each $\mathcal{T}_{p^2}$ commutes with $\Delta_{\frac{1}{2}}$, so $\mathcal{T}_{p^2}$ is an endomorphism of $\tilde{\mathcal{L}}_{\frac{1}{2}}(1,\chi,r)$. The analogous discussion above guarantees that there exists an orthonormal basis of $\tilde{\mathcal{L}}_{\frac{1}{2}}(1,\chi)$ consisting of Hecke eigenforms.

Ahlgren and Andersen \cite{AA} developed a Shimura type correspondence between Maass cusp forms of weight $1/2$ on $\Gamma_0(N)$ with the eta multiplier twisted by a Dirichlet character and Maass cusp forms of weight $0$. Here we provide details only in the simplest case. In this setting it is most convenient to write the expansion of $f \in \tilde{\mathcal{L}}_{\frac{1}{2}}(1,\chi,r)$ in the form 
\begin{equation} \label{normalise}
f(\tau)=\sum_{n \neq 0} a(n) W_{\frac{\text{sgn}(n)}{4},ir} \Big( \frac{\pi |n| y}{6} \Big) e \Big( \frac{nx}{24} \Big).
\end{equation}

\begin{theorem} \cite[Theorem~5.1]{AA} \label{thelif}
Suppose that $G \in \tilde{\mathcal{L}}_{\frac{1}{2}}(1,\chi,r)$ with $r \neq i/4$ and Fourier expansion given by \eqref{normalise}. Let $t \equiv 1 \pmod{24}$ be a square-free positive integer and define $b_t(n) \in \mathbb{C}$ by the relation
\begin{equation} \label{shimrel}
\sum_{n=1}^{\infty} \frac{b_t(n)}{n^s}=L\Big (s+1, \Big( \frac{t}{\bullet} \Big)   \Big) \sum_{n=1}^{\infty} \Big( \frac{12}{n} \Big) \frac{a(tn^2)}{n^{s-\frac{1}{2}}}. 
\end{equation}
Then the function $S_t(G)$ defined by
 \begin{equation*}
(S_t G)(\tau):=\sum_{n=1}^{\infty} b_t(n) W_{0,2ir}(4 \pi n y) \cos( 2 \pi n x)
\end{equation*}
is a Maass cusp form in $\tilde{\mathcal{L}}_0(6,\mathbf{1},2r)$. For any prime $p \geq 5$ we have
 \begin{equation*}
 \mathcal{T}_p S_t(G)= \Big(\frac{12}{p} \Big) S_t \big( \mathcal{T}_{p^2} G \big).
 \end{equation*} 
\end{theorem}

\begin{remark} \label{specremark}
Using Theorem \ref{thelif}, Ahlgren and Andersen rule out the existence of exceptional eigenvalues in $\tilde{\mathcal{L}}_{\frac{1}{2}}(1,\chi)$. If $\tilde{\mathcal{L}}_{\frac{1}{2}}(1,\chi,r) \neq \{0\}$, then either $r=i/4$ or $r>1.9$. Note that $r_0=i/4$ corresponds to the minimal eigenvalue $\lambda_0=\frac{3}{16}$. This is achieved by the unique normalised cusp form 
\begin{equation*}
u_0(\tau):=\sqrt{\frac{3}{\pi}} (6y)^{\frac{1}{4}} \eta(\tau).
\end{equation*}
The Fourier coefficients $\rho_0(m)$ of $u_0$ are non-zero only when $m-1 \in \mathcal{P}$. See \cite[pg.~435]{AA2}.
\end{remark}

\section{Kuznetsov--Proskurin formula} \label{proskurin}
Here we develop some tools for the case $m,n>0$. Let $\phi \in C^{4} \big([0,\infty) \big)$ be such that 
\begin{equation} \label{phicond}
\phi(0)=\phi^{\prime}(0), \quad \phi(t) \ll_{\varepsilon} t^{-1-\varepsilon} \quad \text{and} \quad \phi^{(j)} \ll_{\varepsilon} t^{-2-\varepsilon} \quad \text{for} \quad j=1,2,3,4,
\end{equation}
as $t \rightarrow \infty$ for some fixed $\varepsilon>0$. We define the auxiliary integrals 
\begin{equation} \label{orig}
\check{\phi}(r):=\int_{0}^{\infty} J_{r-1}(y) \phi(y) \frac{dy}{y},
\end{equation}
and
\begin{equation} \label{hat}
\hat{\phi}(r):= \pi^2 e^{3 \pi i/4} \frac{\int_{0}^{\infty} \Big( \cos \pi \big(\frac{1}{4} +ir \big) J_{2ir}(y)-\cos \pi \big( \frac{1}{4}-ir \big) J_{-2 i r}(y) \Big) \phi(y) \frac{dy}{y} }{\text{sh}(\pi r)
\text{ch} (2 \pi r) \Gamma \big(\frac{1}{4}+ir \big) \Gamma \big(\frac{1}{4}-ir \big)},
\end{equation}
where $J_{\nu}$ for $\nu \in \mathbb{C}$ denotes the $J$--Bessel function \cite[Section~10.2]{DL}.

Using a trigonometric identity, we write the integrand occurring in $\hat{\phi}(r)$ in the more convenient form 
\begin{equation} \label{con}
\frac{1}{\sqrt{2}} \frac{\phi(y)}{y} \Big( \cos( \pi i r) \big(J_{2ir}(y)-J_{-2ir}(y) \big)-\sin( \pi i r) \big(J_{2ir}(y)+J_{-2ir}(y) \big)    \Big).
\end{equation}

Endow $S_{\frac{1}{2}+2l}(1,\chi)$ with the usual inner-product \cite[pg~2514]{DFI2}. For each integer $l \geq 1$, let $B_l$ denote an orthonormal basis for $S_{\frac{1}{2}+2l}(1,\chi)$ and 
\begin{equation*}
\mathcal{S}:=\bigcup_{l=1}^{\infty} B_l.
\end{equation*}
Suppose each $f \in \mathcal{S}$ has Fourier expansion given by
\begin{equation*}
f(\tau):=\sum_{n=1}^{\infty} a_f(n) e \big (\tilde{n} \tau \big),
\end{equation*}
and weight denoted by $w(f)$. Let $\{u_j\}$ be an orthonormal basis for $\tilde{\mathcal{L}}_{\frac{1}{2}}(1,\chi)$ with Fourier expansion given by \eqref{fouriercusp}. For $m,n>0$, Proskurin's formula \cite[p.~3888]{P} asserts that
\begin{equation} \label{PrKu}
\sum_{c \geq 1} \frac{S(m,n,c,\chi)}{c} \phi \Big(\frac{4 \pi \sqrt{\tilde{m} \tilde{n}}}{c}  \Big)=\mathcal{U}+\mathcal{V},
\end{equation}
where
\begin{align}
\mathcal{U}&:=\sum_{f \in \mathcal{S}} \frac{4 \Gamma \big(w(f) \big) e^{\pi i w(f) / 2}}{(4 \pi)^{w(f)} (\tilde{m} \tilde{n})^{(w(f)-1)/2} } \overline{a_f(m)} a_f(n) \check{\phi} \big(w(f) \big) \label{U}, \\
\mathcal{V}&:=4 \sqrt{\tilde{m} \tilde{n}} \sum_{j \geq 0} \frac{\overline{\rho_j(m)} \rho_j(n)}{\text{ch} (\pi r_j)} \hat{\phi}(r_j) \label{V}.
\end{align}

Given $a,x>0$, choose a parameter $T>0$ such that 
\begin{equation*}
T \leq x/3, \quad T \asymp x^{1-\delta} \quad \text{with} \quad 0<\delta<1/2.
\end{equation*}
Now we choose a smooth $\phi=\phi_{a,x,T}: [0,\infty) \rightarrow [0,1]$ satisfying 
\begin{itemize}
\item $\phi(t)=1$ for $\frac{a}{2x} \leq t \leq \frac{a}{x}$ 
\item $\phi(t)=0$ for $t \leq \frac{a}{2x+2T}$ and $t \geq \frac{a}{x-T}$
\item $\phi^{\prime}(t) \ll \big(\frac{a}{x-T}-\frac{a}{x} \big)^{-1} \ll  \frac{x^2}{aT}$
\item $\phi$ and $\phi^{\prime}$ are piecewise monotone on a fixed number of intervals.
\end{itemize}

Here we provide bounds for some useful expressions involving $\check{\phi}$ and $\hat{\phi}$.

\begin{lemma} \label{int2}
Let $\phi=\phi_{a,x,T}$ be as above. For $a:=4 \pi \sqrt{ \tilde{m} \tilde{n}}$ we have 
\begin{equation} \label{trans2}
\sum_{l=1}^{\infty} \Big({-\frac{1}{2}+2l }\Big) \Big | \check{\phi} \Big(\frac{1}{2}+2l \Big) \Big | \ll 1+\frac{\sqrt{mn}}{x}.
\end{equation}
\end{lemma}
\begin{proof}
For the reader's convenience we sketch the argument that appears on \cite[pp.~630--632]{ST}, indicating what differs in our situation. When $x \geq 4 \pi \sqrt{\tilde{m} \tilde{n}}$, the support of $\phi$ is contained in $[0,3/2]$, and it is immediate from the decay of the Bessel function \cite[(10.14.4)]{DL} that
\begin{equation*}
\check{\phi} \Big(\frac{1}{2}+2l \Big) \ll \frac{1}{\Gamma(\frac{1}{2}+2l)}.
\end{equation*}
Thus the left hand side of \eqref{trans2} is bounded by $O(1)$.  

Now consider the case when $x \leq 4 \pi \sqrt{\tilde{m} \tilde{n}}$. In what follows we write $k:=1/2+2l$ for convenience. We treat each integral \eqref{orig} occurring in the summand of \eqref{trans2} according to the transitional ranges of the $J$-Bessel function. In the range $0 \leq y \leq k-k^{\frac{1}{3}}$, $J_k(y)$ is exponentially small and the contribution is $O(1)$. 

We now consider the transitional range $k-k^{\frac{1}{3}} \leq y \leq k +k^{\frac{1}{3}}$. A computation using \cite[(10.20.2)--(10.20.4),(9.6.1),(9.6.2),(9.6.6) and (9.6.7)]{DL} establishes  the asymptotics for $J_k(y)$ asserted in \cite[(22) and (23)]{ST} when $k$ is half-integral and positive. Using these asymptotics and the fact that the support of $\phi$ is contained in the interval $\big [0, 6 \pi \sqrt{mn} /x \big]$, we see that there are at most $O(\sqrt{m n} /x)$ choices of $k$ for which the transitional range is present on the left hand side of \eqref{trans2}. Each $(k-1) \check{\phi}(k)$ is $O(1)$ for all $k$ in the transitional range using the asymptotics for $J_{\pm 1/3}$ and $K_{1/3}$ in \cite[(10.7.3) and (10.30.2)]{DL} and hence the total contribution from all such $k$ to \eqref{trans2} is $O(\sqrt{mn}/x)$. 

We are now left to bound the contribution for the range $y \geq k+k^{\frac{1}{3}}$. For this one can follow the argument in \cite[pp.~631--632]{ST} starting with the asymptotic in \cite[Eqn~(52)]{ST}. The contribution in this last case is $O(\sqrt{mn} /x)$.
\end{proof}

\begin{lemma} \label{int}
Suppose that $a,x,T$ are as above and that $\phi=\phi_{a,x,T}$. Then we have
\begin{equation*} 
\hat{\phi}(r) \ll
\begin{cases}
\min \big ( r^{-1}, r^{-2} \frac{x}{T} \big ) & \quad  \text{if}  \quad r \geq \max \big(\frac{a}{x},1 \big) \\
r^{-1} & \quad \text{if} \quad r \geq 1.
\end{cases}
\end{equation*}
\end{lemma}
\begin{proof}
In view of \eqref{hat} and \eqref{con}, it is sufficient to bound $\check{\phi}(2ir+1)$. Sarnak and Tsimerman \cite[pp~629--630]{ST} prove that
\begin{equation*}
\frac{\text{ch}(\pi r )}{\text{sh}(2 \pi r)} \big | \check{\phi}(2ir+1) \big | \ll
\begin{cases}
 r^{-\frac{3}{2}} & \text{for} \quad r \geq 1 \\
\min \big ( r^{-\frac{3}{2}}, r^{-\frac{5}{2}} \frac{x}{T} \big ) & \text{for} \quad r \geq  \max \big(a/x,1 \big). 
\end{cases}
\end{equation*}
The result follows by recalling the definition of $\hat{\phi}(r)$ in \eqref{hat} and observing that
\begin{equation*}
\frac{1}{\big | \Gamma( \frac{1}{4}+ir) \big |^2} \sim \frac{\sqrt{r}}{2 \pi} e^{\pi r}  \quad \text{as} \quad r \rightarrow \infty,
\end{equation*}
by Stirling's formula.
\end{proof}

\section{Variant of Proskurin--Kuznetsov formula} \label{AAkuzsec}
We introduce the tools for the mixed sign case $m>0$ and $n<0$. Let $\phi$ be as in Section~\ref{proskurin}. Define 
\begin{equation*}
\check{\Phi}(r):=\cosh(\pi r) \int_{0}^{\infty}  K_{2ir}(y) \phi(y) \frac{dy}{y},
\end{equation*}
where $K_{\nu}$ for $\nu \in \mathbb{C}$ denotes the $K$--Bessel function \cite[Section~10.25]{DL}. Let $\{u_j\}$ be an orthonormal basis for $\tilde{\mathcal{L}}_{\frac{1}{2}}(1,\chi)$ with Fourier expansions given by \eqref{fouriercusp}. Then for $m>0$ and $n<0$, \cite[Theorem~4.1]{AA} asserts that
\begin{equation} \label{AAkuz}
\mathcal{W}:=\sum_{c \geq 1} \frac{S(m,n,c,\chi)}{c} \phi \Big(\frac{4 \pi \sqrt{\tilde{m} |\tilde{n}|}}{c}  \Big)=8 \sqrt{i} \sqrt{\tilde{m} |\tilde{n}|} \sum_{j \geq 0} \frac{\overline{\rho_j(m)} \rho_j(n)}{\text{ch} (\pi r_j)} \check{\Phi}(r_j).
\end{equation}

\begin{lemma} \cite[Theorem~6.1]{AA} \label{int3}
Let a,x,T be as above and let $\phi=\phi_{a,x,T}$. Then
\begin{equation*}
\check{\Phi}(r) \ll \begin{cases}
 r^{-\frac{3}{2}} e^{-\frac{r}{2}} & \emph{for} \quad 1 \leq r \leq \frac{a}{8x} \\
 r^{-1} & \emph{for} \quad \max \big(1,\frac{a}{8x} \big) \leq r \leq \frac{a}{x}   \\
 \min \big(r^{-\frac{3}{2}}, r^{-\frac{5}{2}} \frac{x}{T} \big ) & \emph{for} \quad r \geq  \max \big( \frac{a}{x},1 \big). 
\end{cases}
\end{equation*}
\end{lemma}

\section{Bound for holomorphic forms}  \label{holbd}
In this section we bound the $\mathcal{U}$ term in \eqref{PrKu} uniformly in $m,n$ and $x$. To obtain bounds in terms of the square-free parts of $24m-23$ and $24n-23$, we exploit the Shimura correspondence and Deligne's bound. 

For our purposes, it is sufficient to consider $\Psi$ an even Dirichlet character mod $4N$ and $k \in 2 \mathbb{N}$. Endow  $S_{\frac{1}{2}+k}(4N,\Psi \nu_{\theta})$ with the usual inner product \cite[pg~2514]{DFI2}. We first recall the half--integral weight Petersson formula \cite[Lemma~4]{Blo} that will apply in this setting. Let $\big \{ \psi_j:=\sum_{r \geq 1} a_j(n) e(n \tau) \big \}_{j=1}^J$ be an orthonormal basis for $S_{\frac{1}{2}+k}
(4N,\Psi \nu_{\theta})$. Then for $k \geq 2$ we have 
\begin{equation} \label{pet}
\frac{\Gamma(k-\frac{1}{2})}{(4 \pi n)^{k-\frac{1}{2}}} \sum_{j=1}^J |a_j(n)|^2=1+2 \pi i^{-\frac{1}{2}-k} \sum_{4N \mid c} c^{-1} J_{k-\frac{1}{2}} \Big(\frac{4 \pi n}{c} \Big)  K_{\Psi}(n,n,c),
\end{equation}
where 
\begin{equation*}
K_{\Psi}(m,n,c):=\sum_{d \pmod c} \varepsilon_d \Big( \frac{c}{d} \Big) \Psi(d) e \Big( \frac{m d+n \overline{d}}{c} \Big), 
\end{equation*}
is a twisted Kloosterman sum.

\begin{lemma} \label{Ubd}
Suppose $m,n>0$ are integers and $\phi$ is as above with $x \geq 1$ and $a:=4 \pi \sqrt{\tilde{m} \tilde{n}}$. Then for any $\varepsilon>0$ we have 
\begin{equation} \label{normalholo}
\mathcal{U} \ll_{\varepsilon} (mn)^{\varepsilon} \Big( (mn)^{\frac{1}{4}} + \frac{(m n)^{\frac{3}{4}}}{x} \Big).
\end{equation}
Furthermore, if $24 m-23=m_0^2 s$ and $ 24n-23=n_0^2 t$ with $s$ and $t$ square-free, then 
\begin{equation} \label{holohecke}
\mathcal{U} \ll_{\varepsilon} |m_0 n_0|^{\varepsilon} (st)^{\frac{1}{4}+\varepsilon} \Big(1+\frac{(mn)^{\frac{1}{2}}}{x} \Big).
\end{equation}
\end{lemma}
\begin{proof}
Let $\{f_{jl} \}_{1 \leq j \leq \text{dim} S_{\tiny{\frac{1}{2}}+2l}(1,\chi)}$ be an orthonormal Hecke eigenbasis with respect to $\tilde{T}_{p^2}$ for all primes $p \nmid 6$. Suppose that 
\begin{equation*}
f_{jl}(\tau)=\sum_{r=1}^{\infty} a_{jl}(r) e \Big(\Big(r-\frac{23}{24} \Big) \tau \Big)
\end{equation*} 
and let 
\begin{equation*}
g_{jl}(\tau):=L(f_{jl})(\tau)=\sum_{m=1}^{\infty} c_{jl}(m) e (m \tau) \in S_{\frac{1}{2}+2l}(576,\chi_{12} \nu_{\theta}).
\end{equation*}
For fixed $l$, the set $\{f_{jl}\}$ injects into its image $\{g_{jl}\} \subset S_{\frac{1}{2}+2l}(576,\chi_{12} \nu_{\theta})$ and $\{g_{jl}\}$ is an orthogonal set consisting of Hecke eigenforms for all $T_{p^2}$ with $p \nmid 6$. For each $l$, a computation shows that
\begin{equation} \label{rescale}
\{(24)^{l+\frac{1}{4}} [\Gamma:\Gamma_0(24,24)]^{-\frac{1}{2}} g_{jl} \}
\end{equation} 
is an orthonormal set of Hecke eigenforms. Since $a_{jl}(n)=c_{jl}(24n-23)$, we trivially have
\begin{equation} \label{ineq}
\sum_{j} |a_{jl}(n)|^2=\sum_{j} |c_{jl}(24n-23)|^2.
\end{equation}

Applying the triangle and Cauchy--Schwarz inequalities to the right side of \eqref{U} and using \eqref{ineq} we obtain 
\begin{equation} \label{CR}
\mathcal{U} \ll \sum_{l=1}^{\infty} \frac{ \Gamma(\frac{1}{2}+2l) | \check{\phi}(\frac{1}{2}+2l)  | }{(4 \pi)^{\frac{1}{2}+2l} (\tilde{m} \tilde{n})^{-\frac{1}{4}+l}} \Big( \sum_j  | c_{jl}(24m-23)|^2   \Big)^{1/2} \Big( \sum_j | c_{jl}(24n-23) |^2   \Big)^{1/2}. 
\end{equation}
The set in \eqref{rescale} can be extended to an orthonormal basis of $S_{\frac{1}{2}+2l}(576,\chi_{12} \nu_{\theta} )$. Applying \eqref{pet} and the triangle inequality we obtain 
\begin{multline} \label{petersson}
\sum_{j} |c_{jl}(24n-23)|^2 \ll \frac{(24)^{-2l} \Big(4 \pi (24n-23) \Big)^{-\frac{1}{2}+2l}}{\Gamma(-\frac{1}{2}+2l)} \\
\times  \Bigg(1+2 \pi \sum_{c \equiv 0 \pmod{576}} \frac{\big| K_{\chi_{12}}(24n-23,24n-23,c) \big|}{c} \Big| J_{-\frac{1}{2}+2l} \Big(\frac{4 \pi (24 n-23)}{c} \Big) \Big|   \Bigg).
\end{multline}
Let $\delta>0$ be fixed and small. To bound the right hand side of \eqref{petersson} we consider the cases $c \leq n^{1+\delta}$ and $c> n^{1+\delta}$. Here we stress that the estimates obtained are uniform in $l$. When $c \leq n^{1+\delta}$, we use \cite[Lemma~4]{Wai} 
\begin{equation} \label{kloost}
|K_{\chi_{12}}(24n-23,24n-23,c) | \leq (24n-23,c)^{\frac{1}{2}} \tau(c) c^{\frac{1}{2}},
\end{equation}
together with \cite[(10.14.1)]{DL} to obtain
\begin{equation} \label{initseg}
 \sum_{\substack{1 \leq c \leq n^{1+\delta} \\ c \equiv 0 \pmod{576}}} \frac{\big| K_{\chi_{12}}(24n-23,24n-23,c) \big|}{c} \Big |  J_{-\frac{1}{2}+2l}  \Big(\frac{4 \pi (24n-23)}{c} \Big) \Big | \ll_{\varepsilon} n^{\frac{1}{2}+\frac{\delta}{2}+\varepsilon}.
\end{equation}
In the case $c \geq n^{1+\delta}$ we apply \cite[(10.14.4)]{DL} and \eqref{kloost} in the following computation
\begin{align}
 \sum_{\substack{c>n^{1+\delta} \\ c \equiv 0 \pmod{576}}} & \frac{\big|K_{\chi_{12}}(24n-23,24n-23,c) \big|}{c} \Big| J_{-\frac{1}{2}+2l} \Big(\frac{4 \pi (24n-23)}{c} \Big) \Big | \nonumber \\
 & \ll \frac{1}{\Gamma(-\frac{1}{2}+2l)}  \sum_{c>n^{1+\delta}} \frac{|K_{\chi_{12}}(24n-23,24n-23,c)|}{c} \Big( \frac{4 \pi (24n-23)}{2c} \Big)^{-\frac{1}{2}+2l} \nonumber \\
& \ll_{\varepsilon} \frac{1}{ \Gamma(-\frac{1}{2}+2l)} \sum_{c>n^{1+\delta}} \frac{1}{c^{\frac{1}{2}-\varepsilon}} (24n-23,c)^{\frac{1}{2}} \Big( \frac{4 \pi (24n-23)}{2c} \Big)  \Big( \frac{4 \pi (24n-23)}{2c} \Big)^{-\frac{3}{2}+2l} \nonumber  \\
& \ll_{\varepsilon} \frac{(48 \pi)^{2l} n^{1-\delta(-\frac{3}{2}+2l)} }{ \Gamma(-\frac{1}{2}+2l)} \sum_{d \mid 24n-23} \frac{1}{d^{1-\varepsilon}} \sum_{c^{\prime}>n^{1+\delta}/d} \frac{1}{(c^{\prime})^{\frac{3}{2}-\varepsilon}} \nonumber \\
& \ll_{\varepsilon} n^{\frac{1}{2}+\varepsilon}  \label{case2}. 
\end{align}
Combining \eqref{CR}--\eqref{case2} we obtain
\begin{equation} \label{trans}
\mathcal{U} \ll_{\varepsilon} (mn)^{\frac{1}{4}+\varepsilon} \sum_{l=1}^{\infty} \Big({-\frac{1}{2}+2l} \Big) \Big | \check{\phi} \Big(\frac{1}{2}+2l \Big) \Big |. 
\end{equation}
Thus \eqref{normalholo} follows from Lemma \ref{int2}.

We now prove \eqref{holohecke}. Since $g_{jl} \in S_{\frac{1}{2}+2l}(576,\chi_{12} \nu_{\theta})$ is an eigenform under the action of $T_{n^2}$ for all $n$ coprime to $6$, we know that $\text{Sh}_t(g_{jl}) \in S_{4l}(288,\mathbf{1})$ is an eigenform under the action of $T_{n}$ with the same eigenvalue. Denote this eigenvalue  by $\lambda_{jl}(n)$. For each $l$ and $j$ define $b_t(n) \in \mathbb{C}$ by 
\begin{equation*}
\text{Sh}_{t}(g_{jl}) (\tau)=\sum_{n=1}^{\infty} b_t(n) e(n \tau) \in S_{4l}(288,\mathbf{1}).
\end{equation*}
We also define the arithmetic functions
\begin{align}
g(u):=c_{jl}(t u^2) \quad \text{and} \quad h(u):=u^{2l-1} \Big( \frac{12t}{u} \Big).
\end{align}
The equality in Lemma \ref{shim} implies that $b_t=g * h$. Observe that $h$ is multiplicative and $h(1)=1$, so $h$ has a multiplicative Dirichlet inverse. We have $h^{-1}(1)=1$ and a computation for $ p$ prime and $\alpha \in \mathbb{N}$ yields 
\begin{equation*}
h^{-1}(p^{\alpha})=\begin{cases}
-\big(\frac{12t}{p}  \big) p^{2l-1} & \quad \text{if} \quad \alpha=1 \\
0 & \quad \text{if} \quad \alpha \geq 2. 
\end{cases}
\end{equation*}
Thus 
\begin{equation} \label{invh}
|h^{-1}(u)| \leq u^{2l-1} \quad \text{for all} \quad u \in \mathbb{N}.
\end{equation}
We now use the fact that the $\text{Sh}_t(g_{jl})$ are Hecke eigenforms for the $T_n$ with $(n,6)=1$  and are of integral weight. We have $b_t(1)=c_{jl}(t)$, so for $(d,6)=1$ we have
\begin{equation*}
b_t (d)=\lambda_{jl}(d) b_t(1)=\lambda_{jl} (d) c_{jl}(t).
\end{equation*}
Thus
\begin{equation} \label{cjl}
c_{jl}(tn_0^2)=\sum_{d \mid n_0} b_t(d) h^{-1} \Big( \frac{n_0}{d} \Big)=c_{jl}(t) \sum_{d \mid n_0} \lambda_{jl}(d) h^{-1} \Big( \frac{n_0}{d} \Big).
\end{equation}
By Deligne's bound \cite{De} we have 
\begin{equation} \label{delbd}
|\lambda_{jl}(d)| \ll_{\varepsilon} d^{2l-\frac{1}{2}+\varepsilon}.
\end{equation}
Using \eqref{invh} and \eqref{delbd}, \eqref{cjl} becomes 
\begin{equation*}
|c_{jl}(24n-23)|=|c_{jl}(tn_0^2)|  \ll_{\varepsilon} |c_{jl}(t)| n_0^{2l-\frac{1}{2}+\varepsilon}.
\end{equation*}

We may replace each summand on the right hand side of \eqref{ineq} with $|c_{jl}(t) | n_0^{2l-\frac{1}{2}+\varepsilon}$. Performing similar computations to those occuring in \eqref{CR}--\eqref{case2}, we see that \eqref{trans} becomes 
\begin{equation*}
\mathcal{U} \ll_{\varepsilon} (m_0 n_0)^{\varepsilon} (st)^{\frac{1}{4}+\varepsilon} \sum_{l=1}^{\infty} \Big({-\frac{1}{2}+2l} \Big) \Big | \check{\phi} \Big(\frac{1}{2}+2l \Big) \Big |.
\end{equation*}
Thus Lemma \ref{int2} implies \eqref{holohecke}.
\end{proof}

\section{Estimates for the coefficients of Maass cusp forms} \label{massbd}
We bound the quantities in \eqref{V} and \eqref{AAkuz} in the proofs of the main theorems by modifying the dyadic arguments of \cite{ST}.  Here we collect the main inputs required for this argument. 

\begin{prop} \cite[Theorem~1.5]{AA} \label{AAmve}
Suppose $\{u_j\}$ is an orthonormal basis for $\tilde{\mathcal{L}}_{\frac{1}{2}}(1,\chi)$ with spectral parameters $r_j$ and Fourier expansion given by \eqref{fouriercusp}. If $n<0$ then we have
\begin{equation*}
|\tilde{n}| \sum_{0<r_j \leq x} \frac{|\rho_j(n)|^2}{\emph{ch}(\pi r_j)}=\frac{x^{\frac{5}{2}}}{5 \pi^2}+ O_{\varepsilon} \big( x^{\frac{3}{2}} 
\log x+ |n|^{\frac{1}{2}+\varepsilon} x^{\frac{1}{2}} \big).
\end{equation*}
\end{prop}

Andersen and Duke improve the mean value estimate of \cite[Theorem~1.5]{AA} when $n>0$. This improvement will be important in the proofs of our main theorems. We state only a special case of their result.
\begin{prop} \cite[Theorem~4.1]{And} \label{andersen}
Suppose $\{u_j\}$ is an orthonormal basis for $\tilde{\mathcal{L}}_{\frac{1}{2}}(1,\chi)$ with spectral parameters $r_j$ and Fourier expansion given by \eqref{fouriercusp}. If $n>0$ then we have
\begin{equation*}
\tilde{n} \sum_{x \leq r_j \leq 2x} \frac{|\rho_j(n)|^2}{\emph{ch}(\pi r_j)} \ll_{\varepsilon} x^{\frac{3}{2}}+x^{-\frac{1}{2}} (\log x)^{\frac{1}{2}+\varepsilon}  n^{\frac{1}{2}+\varepsilon}.
\end{equation*}
\end{prop}

One can relate the positively (resp. negatively) indexed coefficients of a Maass cusp form of weight $\frac{1}{2}$ to negatively (resp. positively) indexed coefficients of a Maass cusp form of weight $-\frac{1}{2}$ via the conjugation map (cf. \cite[(2.6)]{AD}). This accounts for difference in the order of magnitude for the $x$ parameter in the estimates occurring in Propositions \ref{AAmve} and \ref{andersen}.

The second main idea is the application of an averaged form of a pointwise bound due to Duke \cite{D} for the Fourier coefficients of Maass cusp forms of half integral weight with multiplier $\big(\frac{12}{\bullet} \big) \nu_{\theta}$. This average bound was established by Ahlgren and Andersen \cite[Theorem~8.1]{AA} using a modified version of Duke's argument. When $n>0$, we remove a hypothesis in their Theorem which restricts divisibility of $n$ by arbitrarily high powers of $5$ and $7$. 

\begin{prop} \label{hyperem}
Let $\{u_j\}$ be an orthonormal basis for $\tilde{\mathcal{L}}_{\frac{1}{2}}(1,\chi)$ with spectral parameters $r_j$ and Fourier expansion given by \eqref{fouriercusp}. If all the $u_j$ are Hecke eigenforms of $\mathcal{T}_{p^2}$ with $p \nmid 6$, then for $x \geq 1$ and $n>0$ we have, 
\begin{equation*} 
n \Big | \sum_{0<r_j \leq x} \frac{|\rho_j(n)|^2}{\emph{ch}(\pi r_j)} \Big | \ll_{\varepsilon}  n^{\frac{3}{7}+\varepsilon} x^{\frac{9}{2}}.
\end{equation*}
\end{prop}
\begin{proof}
Suppose $24n-23=n_0^2 t$ with $t$ square-free. We first prove
\begin{equation} \label{replace}
|\rho_j(n) | \ll_{\varepsilon} \Big | \rho_j \Big(\frac{t+23}{24} \Big) \Big | n_0^{-1+\theta+\varepsilon}.
\end{equation}
Recalling the normalisations \eqref{fouriercusp} and \eqref{normalise} we define
\begin{equation} \label{fourierrel}
a_j(n)=\rho_j \Big(\frac{n+23}{24} \Big).
\end{equation}
We write 
\begin{equation*}
S_t(u_j):=\sum_{n=1}^{\infty} b_t(n,j) W_{0,2ir}(4 \pi n y) \cos( 2 \pi n x) \in \tilde{\mathcal{L}}_0(6,\mathbf{1},2r_j),
\end{equation*}
where $S_t$ denotes the lift for cusp forms in Theorem \ref{thelif}. The $S_t(u_j)$ are also eigenforms under the action of $\mathcal{T}_n$ with $(n,6)=1$. Let the eigenvalue be denoted by $\lambda_j(n)$. We define the arithmetic functions
\begin{equation*}
g(u):=a_{j}(t u^2) u^{\frac{1}{2}} \Big( \frac{12}{u} \Big) \quad \text{and} \quad h(u):=u^{-1} \Big( \frac{t}{u} \Big).
\end{equation*}
The equality in Theorem \ref{thelif} implies that $b_t(\cdot,j)=g * h$. Observe that $h(1)=1$, so arguing as in Section~\ref{holbd} gives 
\begin{equation*}
h^{-1}(p^{\alpha})=\begin{cases}
-\big(\frac{t}{p}  \big) p^{-1} & \quad \text{if} \quad \alpha=1 \\
0 & \quad \text{otherwise}. 
\end{cases}
\end{equation*}
Thus 
\begin{equation} \label{invh2}
|h^{-1}(u)| \leq u^{-1} \quad \text{for all} \quad u \in \mathbb{N}.
\end{equation}
Using \eqref{coeff} and the relation $b_t(1,j)=a_{j}(t)$, we have
\begin{equation*}
b_t (d,j)= d^{-\frac{1}{2}} \lambda_{j}(d) b_t(1,j)=d^{-\frac{1}{2}} \lambda_{j} (d) a_{j}(t).
\end{equation*}
Thus 
\begin{equation*}
g_j(tn_0^2)=\sum_{d \mid n_0} b_t(d,j) h^{-1} \Big( \frac{n_0}{d} \Big)=a_{j}(t) \sum_{d \mid n_0} d^{-\frac{1}{2}} \lambda_{j}(d) h^{-1} \Big( \frac{n_0}{d} \Big).
\end{equation*}
Applying the $H_{\theta}$--hypothesis and \eqref{invh2} we obtain
\begin{equation*}
|g_j(tn_0^2)| \ll_{\varepsilon} |a_j(t)| n_0^{-\frac{1}{2}+\theta+\varepsilon}.
\end{equation*}
By the definition of $g$ and \eqref{fourierrel}, we conclude that \eqref{replace} holds.

Under the $H_{\theta}$ hypothesis we can apply \eqref{replace}, followed by \cite[Theorem~8.1]{AA} to obtain
\begin{equation*} 
n \Big | \sum_{0<r_j \leq x} \frac{|\rho_j(n)|^2}{\text{ch}(\pi r_j)} \Big | \ll_{\varepsilon}  n_0^{2 \theta+\varepsilon} t \Big | \sum_{0<r_j \leq x} \frac{|\rho_j \big(\frac{t+23}{24} \big)|^2}{\text{ch}(\pi r_j)} \Big | \ll_{\varepsilon}  n_0^{2 \theta+\varepsilon} t^{\frac{3}{7}+\varepsilon} x^{\frac{9}{2}}.
\end{equation*}
Since $\theta=7/64$ is an acceptable exponent \cite[Appendix 2]{Ki} we see that 
\begin{equation*}
n_0^{2 \theta+\varepsilon} t^{\frac{3}{7}+\varepsilon} x^{\frac{9}{2}} \ll (n_0^2)^{\frac{7}{64}+\varepsilon}  t^{\frac{3}{7}+\varepsilon} x^{\frac{9}{2}} \ll_{\varepsilon} n^{\frac{3}{7}+\varepsilon} x^{\frac{9}{2}}.
\end{equation*}
\end{proof}
 This averaged bound allows us to optimise the dyadic argument in the spectral parameter, and is ultimately responsible for the improved $mn$-aspect occurring in Theorem \ref{mainthm}.

\section{Proof of Theorem \ref{mainthm}} \label{mainsec}
\begin{prop} \label{mainprop}
Let $m>0$ and $n<0$ be integers such that $24n-23$ is not divisible by $5^4$ or $7^4$. Then for $x \geq |\tilde{m} \tilde{n}|^{\frac{38}{77}}$ we have
\begin{equation*}
\sum_{x \leq c \leq 2x} \frac{S(m,n,c,\chi)}{c}  \ll_{\varepsilon} \Big( x^{\frac{1}{6}}+ \frac{m^{\frac{1}{4}}+|n|^{\frac{1}{4}}}{|mn|^{\frac{1}{308}}}  +|mn|^{\frac{19}{77}} \Big) |mn|^{\varepsilon} \log^2 x.
\end{equation*}
\end{prop}
We show that Proposition \ref{mainprop} implies Theorem \ref{mainthm}. Considering the sum
\begin{equation} \label{demsum}
\sum_{1 \leq c \leq X} \frac{S(m,n,c,\chi)}{c},
\end{equation}
the initial segment $1 \leq c \leq |\tilde{m} \tilde{n}|^{\frac{38}{77}}$ contributes $O_{\varepsilon}\big(|\tilde{m} \tilde{n}|^{\frac{19}{77}+\varepsilon} \big)$ by \cite[(2.30)]{AA}. One then breaks the interval $|\tilde{m} \tilde{n}|^{\frac{38}{77}} \leq c \leq X$ into $O(\log X)$ dyadic intervals $x \leq c \leq 2x$ with $|\tilde{m} \tilde{n}|^{\frac{38}{77}} \leq x \leq X/2$, and then applies Proposition \ref{mainprop}. 

\begin{proof}[Proof of Proposition~\ref{mainprop}] 
Let $\phi$ be a smooth test function with the properties listed in Section \ref{proskurin}. Fix $a:=4 \pi \sqrt{\tilde{m} |\tilde{n}|}$ and let $T>0$, $0<\beta<1/2$ both be chosen later. Suppose 
\begin{equation*}
x \geq (\tilde{m} |\tilde{n}|)^{\frac{1}{2}-\beta}.
\end{equation*}
Using the Weil bound \cite[Proposition~2.1]{AA} and the mean value bound for the divisor function we have 
\begin{multline} \label{approx}
\Bigg | \sum_{c=1}^{\infty} \frac{S(m,n,c,\chi)}{c} \phi \big(\frac{a}{c} \big)-\sum_{x \leq c \leq 2x} \frac{S(m,n,c,\chi)}{c} \Bigg | \leq \sum_{\substack{x-T \leq c  \leq x \\ 2x \leq c \leq 2x+2T}} \frac{|S(m,n,c,\chi)|}{c} \\
 \ll_{\varepsilon} \frac{T \log x}{\sqrt{x}} |mn|^{\varepsilon}.
\end{multline}
Recall from \eqref{AAkuz} that we have 
\begin{equation} \label{newdyad2}
\sum_{c=1}^{\infty} \frac{S(m,n,c,\chi)}{c} \phi \Big( \frac{a}{c} \Big)=\mathcal{W}.
\end{equation} 

Let $\{u_j\}$ be an orthonormal basis for $\tilde{\mathcal{L}}_{\frac{1}{2}}(1,\chi)$ that are also a Hecke eigenbasis with respect to the $\mathcal{T}_{p^2}$ for all primes $p \nmid 6$. Note that there is no contribution from $r_0=i/4$ since $n$ is negative. For the initial segment $0<r_j \leq 4 \pi(\tilde{m} |\tilde{n}|)^{\beta}$, we apply Proposition \ref{hyperem} (to the $m>0$ variable) and \cite[Theorem~8.1]{AA} (to the $n<0$ variable) to obtain
 \begin{equation} \label{avgduke3}
m |n| \Bigg | \sum_{0<r_j \leq 4 \pi(\tilde{m} |\tilde{n}|)^{\beta}} \frac{|\rho_j(m)|^2}{\text{ch}(\pi r_j)} \Bigg | \cdot \Bigg | \sum_{0<r_j \leq 4 \pi(\tilde{m} |\tilde{n}|)^{\beta}} \frac{|\rho_j(n)|^2}{\text{ch}(\pi r_j)} \Bigg | \ll_{\varepsilon} |mn|^{\frac{3}{7}+10 \beta+\varepsilon}.
\end{equation}

Then applying the Cauchy--Schwarz inequality, \eqref{avgduke3}, Lemma \ref{int3} and the fact that $r_j>1.9$ \cite[Corollary~5.3]{AA} yields the estimate
\begin{equation} \label{avgduke2}
\sqrt{\tilde{m} |\tilde{n}}|  \Bigg | \sum_{0<r_j \leq 4 \pi(\tilde{m} |\tilde{n}|)^{\beta} } \frac{\overline{\rho_j(m)} \rho_j(n)}{\text{ch}(\pi r_j)} \check{\Phi}(r_j) \Bigg | \ll_{\varepsilon} |mn|^{\frac{3}{14}+5 \beta+\varepsilon}.
\end{equation}

We now estimate the contribution to $\mathcal{W}$ from the dyadic intervals $A \leq r_j \leq 2A$ with $A \geq 4 \pi (\tilde{m} |\tilde{n}|)^{\beta}$. Since we are assuming $x \geq (\tilde{m} |\tilde{n}|)^{\frac{1}{2}-\beta}$, we have $A \geq \max \big(a/x,1 \big)$. Using Lemma \ref{int3}, Proposition \ref{AAmve} and Proposition \ref{andersen} we have
\begin{align}
\sqrt{\tilde{m} |\tilde{n}|} &  \Big | \sum_{A \leq r_j \leq 2A} \frac{\overline{\rho_j(m)} \rho_j(n)}{\text{ch}(\pi r_j)} \check{\Phi}(r_j) \Big | \nonumber \\ 
& \ll \min \big (A^{-\frac{3}{2}},A^{-\frac{5}{2}} \frac{x}{T} \big) \Big( m \sum_{A \leq r_j \leq 2A} \frac{|\rho_j(m)|^2}{\text{ch}(\pi r_j)} \Big)^{\frac{1}{2}} \Big( |n| \sum_{A \leq r_j \leq 2A} \frac{|\rho_j(n)|^2}{\text{ch}(\pi r_j)} \Big)^{\frac{1}{2}} \label{CR2} \\
& \ll_{\varepsilon}  \min \big (A^{-\frac{3}{2}},A^{-\frac{5}{2}} \frac{x}{T} \big) \big(A^{\frac{3}{2}}+m^{\frac{1}{2}+\varepsilon} A^{-\frac{1}{2}} |\log A|^{\frac{1}{2}+\varepsilon} \big)^{\frac{1}{2}} \big(A^{\frac{5}{2}}+|n|^{\frac{1}{2}+\varepsilon} A^{\frac{1}{2}} \big)^{\frac{1}{2}} \nonumber  \\
& \ll_{\varepsilon}  \min \big (\sqrt{A},\frac{x}{T \sqrt{A}} \big) \big(1+A^{-1} |\log A| (m^{\frac{1}{4}+\varepsilon}+|n|^{\frac{1}{4}+\varepsilon})+A^{-2} |\log A| |mn|^{\frac{1}{4}+\varepsilon} \big) \label{compbd} \\
& \ll_{\varepsilon} \min \big (\sqrt{A},\frac{x}{T \sqrt{A}} \big) \bigg(1+A^{-\frac{1}{2}} |\log A| \frac{(m^{\frac{1}{4}+\varepsilon}+|n|^{\frac{1}{4}+\varepsilon})}{|mn|^{\frac{\beta}{2}}}+A^{-\frac{1}{2}} |\log A| |mn|^{\frac{1}{4}-\frac{3}{2} \beta+\varepsilon} \bigg) \nonumber.
\end{align}

Summing over dyadic intervals $[2^{j} \cdot 4 \pi (\tilde{m} |\tilde{n}|)^{\beta},2^{j+1} \cdot 4 \pi (\tilde{m} |\tilde{n}|)^{\beta} ]$ with $j=0,1,2,\ldots$ we see that when $x \geq (\tilde{m} |\tilde{n}|)^{\frac{1}{2}-\beta}$, the following holds:
\begin{equation} \label{fullcusp}
\sqrt{\tilde{m} \tilde{n}}   \Big | \sum_{r_j \geq 4 \pi (\tilde{m} |\tilde{n}|)^{\beta} } \frac{\overline{\rho_j(m)} \rho_j(n)}{\text{ch}(\pi r_j)} \check{\Phi}(r_j) \Big | \ll_{\varepsilon} \Big( \sqrt{\frac{x}{T}}+\frac{m^{\frac{1}{4}}+|n|^{\frac{1}{4}}}{|mn|^{\frac{\beta}{2}}}+|mn|^{\frac{1}{4}-\frac{3}{2} \beta} \Big) |mn|^{\varepsilon} \log^2 x.
\end{equation}
Combining \eqref{avgduke2} and \eqref{fullcusp} we obtain 
\begin{equation} \label{fullcusp2}
\mathcal{W} \ll_{\varepsilon} \Big( \sqrt{\frac{x}{T}}+|mn|^{\frac{3}{14}+5 \beta+\varepsilon }+  \frac{m^{\frac{1}{4}}+|n|^{\frac{1}{4}}}{|mn|^{\frac{\beta}{2}}}+|mn|^{\frac{1}{4}-\frac{3}{2} \beta} \Big) |mn|^{\varepsilon} \log^2 x.
\end{equation}
To balance the $x$-aspect of \eqref{approx} and \eqref{fullcusp2} we choose $T:=x^{\frac{2}{3}}$. To balance the $mn$-aspect of \eqref{fullcusp2} and the contribution from the initial segment for $c$, we set $\beta=\frac{1}{154}$. Combining \eqref{approx} and \eqref{fullcusp2} leads to the result. 
\end{proof}.

\section{Proof of Theorem \ref{mainthm2}} \label{mainsec2}
\begin{prop} \label{mainprop2}
Let $m,n>0$ be integers such that $m-1 \not \in \mathcal{P}$ or $n-1 \not \in \mathcal{P}$. Then for $x \geq 4 \pi \sqrt{\tilde{m} \tilde{n}}$ we have
\begin{equation*}
\sum_{x \leq c \leq 2x} \frac{S(m,n,c,\chi)}{c}  \ll_{\varepsilon} \Big( x^{\frac{1}{6}}+(mn)^{\frac{1}{4}} \Big) (mn)^{\varepsilon} \log x.
\end{equation*}
\end{prop}
Proposition \ref{mainprop2} implies Theorem \ref{mainthm2}. The initial segment $1 \leq c \leq 4 \pi \sqrt{\tilde{m} \tilde{n}}$ contributes $O_{\varepsilon} \big((mn)^{\frac{1}{4}+\varepsilon} \big)$ by \cite[(2.30)]{AA} to \eqref{demsum}. One then breaks the interval $4 \pi \sqrt{\tilde{m} \tilde{n}} \leq c \leq X$ into $O(\log X)$ dyadic intervals $x \leq c \leq 2x$ with $4 \pi \sqrt{\tilde{m} \tilde{n}} \leq x \leq X/2$, and applies Proposition~\ref{mainprop2}. 

\begin{proof}[Proof of Proposition~\ref{mainprop2}]
Let $\phi$ be a smooth test function with the properties listed in Section \ref{proskurin}. Fix $a:=4 \pi \sqrt{\tilde{m} \tilde{n}}$ and let $T>0$ be chosen later. Recall that \eqref{approx} holds (for $m,n>0$ as well) and that by \eqref{PrKu} we have 
\begin{equation} \label{newdyad}
\sum_{c=1}^{\infty} \frac{S(m,n,c,\chi)}{c} \phi \Big( \frac{a}{c} \Big)=\mathcal{U}+\mathcal{V}.
\end{equation} 
Note that there is no contribution from $r_0=i/4$ by Remark \ref{specremark} and the hypothesis on $m$ and $n$. 

When $x \geq 4 \pi \sqrt{\tilde{m} \tilde{n}}$, Lemma \ref{Ubd} guarantees 
\begin{equation} \label{Ubd2}
\mathcal{U} \ll_{\varepsilon} (mn)^{\frac{1}{4}+\varepsilon}.
\end{equation}  

Let $\{u_j\}$ be an orthonormal basis for $\tilde{\mathcal{L}}_{\frac{1}{2}}(1,\chi)$. Applying Lemma \ref{int} and Proposition \ref{andersen} we see that for $A \geq 1$ (this is sufficient by Remark \ref{specremark}) we have
\begin{align}
\sqrt{\tilde{m} \tilde{n}} &  \Big | \sum_{A \leq r_j \leq 2A} \frac{\overline{\rho_j(m)} \rho_j(n)}{\text{ch}(\pi r_j)} \hat{\phi}(r_j) \Big | \nonumber \\ 
& \ll \min \big (A^{-1},A^{-2} \frac{x}{T} \big) \Big( m \sum_{j \geq 1} \frac{|\rho_j(m)|^2}{\text{ch}(\pi r_j)} \Big)^{\frac{1}{2}} \Big( n \sum_{j \geq 1} \frac{|\rho_j(n)|^2}{\text{ch}(\pi r_j)} \Big)^{\frac{1}{2}} \label{startpoint2} \\
& \ll_{\varepsilon}  \min \big (A^{-1},A^{-2} \frac{x}{T} \big) \big(A^{\frac{3}{2}}+m^{\frac{1}{2}+\varepsilon} A^{-\frac{1}{2}} |\log A|^{\frac{1}{2}+\varepsilon} \big)^{\frac{1}{2}} \big(A^{\frac{3}{2}}+n^{\frac{1}{2}+\varepsilon} A^{-\frac{1}{2}} |\log A|^{\frac{1}{2}+\varepsilon} \big)^{\frac{1}{2}} \nonumber  \\
& \ll_{\varepsilon}  \min \big (\sqrt{A},\frac{x}{T \sqrt{A}} \big) \big(1+A^{-1} |\log A| (m^{\frac{1}{4}+\varepsilon}+n^{\frac{1}{4}+\varepsilon})+A^{-2} |\log A| (mn)^{\frac{1}{4}+\varepsilon} \big). \label{startpoint}
\end{align}
Summing over dyadic intervals $[A,2A]$ for $A \geq \max \big(a/x,1 \big)=1$ we obtain
\begin{equation} \label{fullcusp4}
\mathcal{V} \ll_{\varepsilon} \Big( \sqrt{\frac{x}{T}}+(mn)^{\frac{1}{4}+\varepsilon} \Big) \log x.
\end{equation}
To balance the $x$-aspect of \eqref{approx} and \eqref{fullcusp4} we choose $T:=x^{\frac{2}{3}}$. Combining \eqref{approx}, \eqref{newdyad}, \eqref{Ubd2} and \eqref{fullcusp4} leads to the result. 
\end{proof}.

\section{Proof of Theorem \ref{auxthm2}} \label{auxsec}
The heart of the argument is to use Theorem \ref{thelif} as a means to access the $H_{\theta}$--hypothesis.

\begin{prop} \label{auxprop2}
Let $m,n>0$ be such that $m-1 \not \in \mathcal{P}$ or $n-1 \not \in \mathcal{P}$. Suppose $m_0,n_0$ be integers such that $24m-23=m_0^2 s$ and $24n-23=n_0^2 t$ with $s$ and $t$ square-free. If $x \geq (st)^{\frac{1}{6}} (mn)^{\frac{1}{3}}$, then under the $H_{\theta}$--hypothesis we have
\begin{equation*}
\sum_{x \leq c \leq 2x} \frac{S(m,n,c,\chi)}{c}  \ll_{\varepsilon} \Big( x^{\frac{1}{6}}+ (st)^{\frac{1}{4}}+(st)^{\frac{1}{12}} (mn)^{\frac{1}{6}}+ (mnst)^{\frac{1}{8}+\frac{\theta}{4}}  \Big) (mn)^{\varepsilon} \log^2 x.
\end{equation*}
\end{prop}
As before, Proposition \ref{auxprop2} implies Theorem \ref{auxthm2} with the choice of initial segment $1 \leq c \leq (st)^{\frac{1}{6}} (mn)^{\frac{1}{3}}$. 
\begin{proof}[Proof of Proposition~\ref{auxprop2}]
Let $\phi$ be a smooth test function with the properties listed in Section \ref{proskurin}. Fix $a:=4 \pi \sqrt{\tilde{m} \tilde{n}}$ and let $T>0$ be chosen later. Let $\{u_j\}$ be an orthonormal basis for $\tilde{\mathcal{L}}_{\frac{1}{2}}(1,\chi)$ consisting of Hecke eigenforms of $\mathcal{T}_{p^2}$ for all primes $p \nmid 6$. We follow the proof of Proposition \ref{mainprop2} to \eqref{newdyad} (note that \eqref{approx} holds for $m,n>0$ as well). 

To bound $\mathcal{V}$, we will treat the spectral parameter separately on different ranges. Note that it is sufficient to consider $r_j \geq 1$ by Remark \ref{specremark}. In view of Lemma \ref{int}, we consider the ranges 
\begin{align*}
1 \leq & r_j \leq a/x, \\
& r_j \geq \max \big(a/x,1 \big). 
\end{align*}

We first consider the case $r_j \geq \max \big( a/x,1 \big)$. For $A \geq \max \big( a/x,1 \big)$, we first prove 
\begin{multline}  \label{bound2}
\sqrt{\tilde{m} |\tilde{n}|}  \Bigg | \sum_{A \leq r_j \leq 2A} \frac{\overline{\rho_j(m)} \rho_j(n)}{\text{ch}(\pi r_j)} \hat{\phi}(r_j) \Bigg | \ll_{\varepsilon} (m_0 n_0)^{\theta+\varepsilon} \min \big \{\sqrt{A},\frac{x}{T  \sqrt{A}}  \big \}  \\
\times \big(1+A^{-1} |\log A| (s^{\frac{1}{4}+\varepsilon}+t^{\frac{1}{4}+\varepsilon})+A^{-2} |\log A| (st)^{\frac{1}{4}+\varepsilon} \big). 
\end{multline}
To prove \eqref{bound2} we start with \eqref{startpoint2}. Using \eqref{replace} in \eqref{startpoint2} and then applying Proposition~\ref{andersen}, we establish \eqref{bound2}.

Combining \eqref{startpoint} and \eqref{bound2}, for $A \geq \max \big( a/x,1 \big)$ we have
\begin{multline} \label{mincomp}
\sqrt{\tilde{m} \tilde{n}}  \Bigg | \sum_{A \leq r_j \leq 2A} \frac{\overline{\rho_j(m)} \rho_j(n)}{\text{ch}(\pi r_j)} \hat{\phi}(r_j) \Bigg | \ll_{\varepsilon}  \min \Big (\sqrt{A},\frac{x}{T \sqrt{A}} \Big) \times \\
\min \Bigg(1+ A^{-1} |\log A|  (m^{\frac{1}{4}+\varepsilon}+n^{\frac{1}{4}+\varepsilon})+A^{-2} |\log A| (mn)^{\frac{1}{4}},\\
 (m_0 n_0)^{\theta+\varepsilon} \Big[1+A^{-1} | \log A | (s^{\frac{1}{4}+\varepsilon}+t^{\frac{1}{4}+\varepsilon})+A^{-2} |\log A| (st)^{\frac{1}{4}+\varepsilon} \Big]  \Bigg).
\end{multline}
Using the facts that for positive $B,C$ and $D$ we have
\begin{equation*}
 \min  \big( B+C,D \big) \leq  \min \big( B,D \big)+ \min \big( C,D \big) \quad \text{and} \quad \min \big( B,C \big) \leq \sqrt{BC},
\end{equation*}
we simplify \eqref{mincomp}. This right side of \eqref{mincomp} is 
\begin{multline} \label{mincompup}
\ll_{\varepsilon} \min \Big (\sqrt{A},\frac{x}{T \sqrt{A}} \Big) \Bigg( 1+ (m_0 n_0)^{\frac{\theta}{2}+\varepsilon} |\log A| \Big[ A^{-1} (m^{\frac{1}{8}+\varepsilon}+n^{\frac{1}{8}+\varepsilon})(s^{\frac{1}{8}+\varepsilon}+t^{\frac{1}{8}+\varepsilon}) \\+A^{-\frac{1}{2}} (m^{\frac{1}{8}+\varepsilon}+n^{\frac{1}{8}+\varepsilon})  +A^{-\frac{3}{2}} (m^{\frac{1}{8}+\varepsilon}+n^{\frac{1}{8}+\varepsilon})(st)^{\frac{1}{8}+\varepsilon} 
+A^{-1} (mn)^{\frac{1}{8}+\varepsilon} \\
+A^{-\frac{3}{2}} (mn)^{\frac{1}{8}+\varepsilon} (s^{\frac{1}{8}+\varepsilon}+t^{\frac{1}{8}+\varepsilon} \big)+A^{-2} (mnst)^{\frac{1}{8}+\varepsilon} \Big] \Bigg).
\end{multline}
Summing over all dyadic intervals $[A,2A]$ with $A \geq \max \big( a/x,1 \big)$ and ignoring the smallest terms we obtain 
\begin{align} 
 \sqrt{\tilde{m} \tilde{n}}  \Big | \sum_{r_j  \geq \max (a/x,1)} \frac{\overline{\rho_j(m)} \rho_j(n)}{\text{ch}(\pi r_j)} \hat{\Phi}(r_j) \Big | & \ll_{\varepsilon} \Big( \sqrt{\frac{x}{T}}+(mnst)^{\frac{1}{8}+\varepsilon} (m_0 n_0)^{\frac{\theta}{2}+\varepsilon} \Big) \log^2 x \nonumber \\
 & \ll_{\varepsilon} \Big( \sqrt{\frac{x}{T}}+(mnst)^{\frac{1}{8}+\frac{\theta}{4}+\varepsilon} \Big) \log^2 x \label{Vapprox2}.
\end{align}

Since $x \geq (mn)^{\frac{1}{3}}$, we have $a/x \leq 4 \pi (mn)^{\frac{1}{6}}$. When $r_j \leq a/x$, Lemma \ref{int} and the fact that $r_j>1.9$ guarantees that we can replace $\min \big(\sqrt{A},x/T \sqrt{A} \big)$ in \eqref{mincomp} and \eqref{mincompup} with $\sqrt{A}$. Thus
\begin{equation} \label{ineq2}
\sqrt{\tilde{m} \tilde{n}}  \Big | \sum_{0<r_j \leq a/x } \frac{\overline{\rho_j(m)} \rho_j(n)}{\text{ch}(\pi r_j)} \hat{\phi}(r_j) \Big | \ll_{\varepsilon} (mnst)^{\frac{1}{8}+\frac{\theta}{4}+\varepsilon}.
\end{equation}

Combining \eqref{Vapprox2} and \eqref{ineq2} we have
\begin{equation} \label{ineq1}
\mathcal{V} \ll_{\varepsilon} \Big( \sqrt{\frac{x}{T}}+ (mnst)^{\frac{1}{8}+\frac{\theta}{4}} \Big) (mn)^{\varepsilon} \log^2 x.
\end{equation}
We choose $T:=x^{\frac{2}{3}}$ to balance \eqref{approx} (with $m,n>0$) and \eqref{ineq1}. When $x \geq (st)^{\frac{1}{6}} (mn)^{\frac{1}{3}}$, Lemma \ref{Ubd} gives
\begin{equation} \label{finalU}
\mathcal{U} \ll_{\varepsilon} (st)^{\frac{1}{4}+\varepsilon} +(st)^{\frac{1}{12}} (mn)^{\frac{1}{6}+\varepsilon}.
\end{equation}
Combining \eqref{approx}, \eqref{ineq1} and \eqref{finalU} finishes the proof. 

\end{proof}

\section{Acknowledgements}
The author thanks Professor Scott Ahlgren for his careful reading of the manuscript and both of the referees for their thorough reports. The author is grateful to Nick Andersen for his insightful comments.

\end{document}